\documentclass[12pt,oneside]{article}
\usepackage{amsmath,amssymb,amsfonts,amsthm}
\newcommand{\T}{{\cal T}}

\newcommand {\cp}{\mathfrak{X}(\pi (M))}

\def\Section#1{\vspace{30truept}\addtocounter{section}{1}\setcounter{thm}{0}
\setcounter{equation}{0}{\noindent\Large\bf
    \arabic{section}.~~#1}\par \vspace{12pt}}

\newtheorem{thm}{Theorem}[section]
\newtheorem{cor}[thm]{Corollary}
\newtheorem{lem}[thm]{Lemma}
\newtheorem{prop}[thm]{Proposition}
\newtheorem{defn}[thm]{Definition}

\newtheorem{rem}[thm]{Remark}

\numberwithin{equation}{section}

\begin{document}
\title{{{\bf On Horizontal Recurrent Finsler Connections}}\footnote{ArXiv: 1706.06079 [math.DG]}}
\author{\bf{ Nabil L. Youssef$^{\,1}$ and A. Soleiman$^{2}$}}
\date{}
\maketitle                     
\vspace{-1.16cm}
\begin{center}
{$^{1}$Department of Mathematics, Faculty of Science, Cairo
University, Giza, Egypt.\\ nlyoussef@sci.cu.edu.eg,\, nlyoussef2003@yahoo.fr}
\end{center}

\begin{center}
{$^{2}$Department of Mathematics, Faculty of Science, Benha
University, Benha,
 Egypt.\\ amr.hassan@fsci.bu.edu.eg,\, amrsoleiman@yahoo.com}
\end{center}

\vspace{0.7cm} \maketitle
\smallskip

\noindent{\bf Abstract.} In this paper we adopt the pullback approach to global Finsler geometry. We investigate horizontally recurrent Finsler connections. We prove that for each scalar ($\pi$)1-form $A$, there exists a unique horizontally recurrent Finsler connection whose $h$-recurrence form is $A$. This result generalizes the existence and uniqueness theorem of Cartan connection. We then study some properties of a special kind of horizontally recurrent Finsler connection, which we call special HRF-connection.



\Section{Introduction}

 The theory of connections is an important field of research of differential
 geometry. It was initially developed to solve pure geometrical
 problems. The most important linear connections in Finsler geometry have been studied by many authors, locally (see for example \cite{r91, Hashiguchi, Hojo, Matsumoto, r93}) and globally (\cite{r61,r21,r22,r92,r94}). In \cite{r92, r94, r96}, we have established  new
proofs of global versions of the existence and
uniqueness theorems for the fundamental linear connections on the
pullback bundle of a Finsler manifold.
\par

In the present paper we investigate a certain type of Finsler connections that generalize Cartan connection; these connections are called horizontally recurrent Finsler (HRF-) connections. We still adopt the pullback formalism
to global Finsler geometry. We prove that for any given scalar ($\pi$)1-form, there exists a unique HRF-connection whose $h$-recurrence form is $A$. We display the associated spray and the associated nonlinear connection. We then study a special kind of such connections which we call special HRF-connection. The results of this paper globalize and generalize some results of \cite{hv, h}.

\Section{Notation and Preliminaries}

In this section, we give a brief account of some basic concepts
 of the pullback approach to intrinsic Finsler geometry necessary for this work. For more
 details, we refer to \cite{ r93, amr3, r94, r96}. We
 shall use the notations of \cite{r94}.

 In what follows, we denote by $\pi: \T M\longrightarrow M$ the slit tangent bundle of $M$, $\mathfrak{F}(TM)$ the algebra of $C^\infty$ functions on $TM$, $\cp$ the $\mathfrak{F}(TM)$-module of differentiable sections of the pullback bundle $\pi^{-1}(T M)$.
The elements of $\mathfrak{X}(\pi (M))$ will be called $\pi$-vector
fields and will be denoted by barred letters $\overline{X} $. The
tensor fields on $\pi^{-1}(TM)$ will be called $\pi$-tensor fields.
The fundamental $\pi$-vector field is the $\pi$-vector field
$\overline{\eta}$ defined by $\overline{\eta}(u)=(u,u)$ for all
$u\in \T M$.
\par
We have the following short exact sequence of vector bundles
$$0\longrightarrow
 \pi^{-1}(TM)\stackrel{\gamma}\longrightarrow T(\T M)\stackrel{\rho}\longrightarrow
\pi^{-1}(TM)\longrightarrow 0 ,\vspace{-0.1cm}$$ with the well known
definitions of  the bundle morphisms $\rho$ and $\gamma$. The vector
space $V_u (\T M)= \{ X \in T_u (\T M) : d\pi(X)=0 \}$  is the vertical space to $M$ at $u\in \T M$.
\par
Let $D$ be  a linear connection on the pullback bundle $\pi^{-1}(TM)$.
 We associate with $D$ the map \vspace{-0.1cm} $K:T \T M\longrightarrow
\pi^{-1}(TM):X\longmapsto D_X \overline{\eta} ,$ called the
connection map of $D$.  The vector space $H_u (\T M)= \{ X \in T_u
(\T M) : K(X)=0 \}$ is called the horizontal space to $M$ at $u$ .
   The connection $D$ is said to be regular if
$$ T_u (\T M)=V_u (\T M)\oplus H_u (\T M) \,\,\,  \forall \, u\in \T M.$$

If $M$ is endowed with a regular connection, then the vector bundle
   maps $
 \gamma,\, \rho |_{H(\T M)}$ and $K |_{V(\T M)}$
 are vector bundle isomorphisms. The map
 $\beta:=(\rho |_{H(\T M)})^{-1}$
 will be called the horizontal map of the connection
$D$.
\par
 The horizontal ((h)h-) and
mixed ((h)hv-) torsion tensors of $D$, denoted by $Q $ and $ T $
respectively, are defined by \vspace{-0.2cm}
$$Q (\overline{X},\overline{Y})=\textbf{T}(\beta \overline{X}\beta \overline{Y}),
\, \,\,\, T(\overline{X},\overline{Y})=\textbf{T}(\gamma
\overline{X},\beta \overline{Y}) \quad \forall \,
\overline{X},\overline{Y}\in\mathfrak{X} (\pi (M)),\vspace{-0.2cm}$$
where $\textbf{T}$ is the (classical) torsion tensor field
of $D$.
\par
The horizontal (h-), mixed (hv-) and vertical (v-) curvature tensors
of $D$, denoted by $R$, $P$ and $S$
respectively, are defined by
$$R(\overline{X},\overline{Y})\overline{Z}=\textbf{K}(\beta
\overline{X}\beta \overline{Y})\overline{Z},\quad
 {P}(\overline{X},\overline{Y})\overline{Z}=\textbf{K}(\beta
\overline{X},\gamma \overline{Y})\overline{Z},\quad
 {S}(\overline{X},\overline{Y})\overline{Z}=\textbf{K}(\gamma
\overline{X},\gamma \overline{Y})\overline{Z}, $$
 where $\textbf{K}$
is the (classical) curvature tensor field of $D$.
\par
The contracted curvature tensors of $D$, denoted by $\widehat{{R}}$, $\widehat{ {P}}$ and $\widehat{ {S}}$ (known
also as the (v)h-, (v)hv- and (v)v-torsion tensors, respectively), are defined by
$$\widehat{ {R}}(\overline{X},\overline{Y})={ {R}}(\overline{X},\overline{Y})\overline{\eta},\quad
\widehat{ {P}}(\overline{X},\overline{Y})={
{P}}(\overline{X},\overline{Y})\overline{\eta},\quad \widehat{
{S}}(\overline{X},\overline{Y})={
{S}}(\overline{X},\overline{Y})\overline{\eta}.$$

\par
Let $(M,L)$ be a Finsler manifold and  $g$ the Finsler metric defined by $L$.
Let $\nabla$ and $D^{\circ}$ be the Cartan and Berwald connections associated with $(M,L)$. We quote the following
results from \cite{r92}.

\begin{prop}\label{prop.1} The Berwald connection $D^{\circ}$ is explicitly expressed in
terms of Cartan connection $\nabla$ in the form:
\begin{description}
   \item[\em{\textbf{(a)}}] $ {{D}}^{\circ}_{\gamma \overline{X}}\overline{Y}=\nabla _{\gamma
  \overline{X}}\overline{Y}-T(\overline{X},\overline{Y})=\rho[\gamma
\overline{X}, \beta \overline{Y}]$,

  \item[\em{\textbf{(b)}}] $ {{D}}^{\circ}_{\beta \overline{X}}\overline{Y}=\nabla _{\beta
  \overline{X}}\overline{Y}+\widehat{P}(\overline{X},\overline{Y})=K[\beta
\overline{X}, \gamma \overline{Y}].$
\end{description}
\end{prop}

\begin{prop} \label{prop.2} The Berwald connection $D^{\circ}$ has the properties:
\begin{description}
\item[(a)] $({{D}}^{\circ}_{\gamma
\overline{X}}\,g)(\overline{Y},\overline{Z}) =2 T(\overline{X},
\overline{Y},\overline{Z})$,

\item[(b)] $({{D}}^{\circ}_{\beta \overline{X}}\,g)(\overline{Y},\overline{Z}) =-2
  g(\widehat{P}(\overline{X},\overline{Y}),\overline{Z})$,
\end{description}
where $ T(\overline{X},
\overline{Y},\overline{Z}):= g(T(\overline{X},\overline{Y}),\overline{Z})$.
\end{prop}

\par We terminate this section
by some concepts and results concerning the Klein-Grifone approach
to intrinsic Finsler geometry. For more details, we refer to
\cite{r21, r22, r27, r97}.
\par

 A semi-spray  is a vector field $X$ on $TM$,
 $C^{\infty}$ on $\T M$, $C^{1}$ on $TM$, such that
$J X = \mathcal{C}$, where $J:=\gamma \circ \rho$ and $\mathcal{C}:=\gamma\overline{\eta}$. A semispray $X$ which is
homogeneous of degree $2$ in the directional argument
($[\mathcal{C},X]= X $ ) is called a spray.

\begin{prop}{\em{\cite{r27}}}\label{spray} Let $(M,L)$ be a Finsler manifold. The vector field
$G$ on $TM$ defined by $i_{G}\Omega =-dE$ is a spray, where
 $E:=\frac{1}{2}L^{2}$ is the energy function and $\Omega:=dd_{J}E$.
 Such a spray is called the canonical spray.
 \end{prop}

A nonlinear connection on $M$ is a vector $1$-form $\Gamma$ on $TM$,
$C^{\infty}$ on $\T M$, $C^{0}$ on $TM$, such that $J \Gamma=J$ and $\Gamma J=-J .$
The horizontal and vertical projectors
 associated with $\Gamma$ are
defined by
   $h:=\frac{1}{2} (I+\Gamma)$ and $v:=\frac{1}{2}
 (I-\Gamma)$, respectively.   The torsion and  curvature of
$\Gamma$ are defined by $t:=\frac{1}{2}[J,\Gamma]$ and
 $\mathfrak{R}:=-\frac{1}{2}[h,h]$,  respectively. A nonlinear
 connection $\Gamma$ is homogenous if $[\mathcal{C}, \Gamma]=0$. It  is conservative
 if $d_{h}E=0$.

\begin{thm} \label{th.9a} {\em{\cite{r22}}} On a Finsler manifold $(M,L)$, there exists a unique
conservative homogenous nonlinear  connection  with zero torsion. It
is given by\,{\em:} \vspace{-0.3cm} $$\Gamma =
[J,G],\vspace{-0.3cm}$$ where $G$ is the canonical spray.\\
 This  nonlinear connection is called the canonical or the Barthel connection  associated with $(M,L)$.
\end{thm}


\Section{Horizontal recurrent Finsler connection}

In this section, we prove that for each $\pi$-form $A$ there exists a unique horizontally recurrent Finsler connection whose $h$-recurrence form is $A$. Throughout, we use the notions and results of \cite{r92}.

\smallskip

\begin{defn}\label{sem.1} Let $D$ be a regular connection on
$\pi^{-1}(TM)$ with horizontal map $\beta$. The semispray  $S:=\beta \overline{\eta}$ will be called the semispray associated with $D$.
The nonlinear connection $\Gamma_D:=2\beta\circ\rho-I$ will
be called the nonlinear connection associated with $D$.
\end{defn}

\begin{lem}\label{eqv.} Let  ${D}$ be a regular connection
on $\pi^{-1}(TM)$ whose connection map is $K$ and whose horizontal
map is $\beta$. The (h)hv-torsion  ${T}$ of ${D}$ has the property that
    ${T}( \overline{X},\overline{\eta})=0$ if and only if the vector form  ${\Gamma} :=\beta\circ\rho - \gamma\circ K$ is a nonlinear
    connection on $M$. In this case ${\Gamma}$
    coincides with the nonlinear connection associated with $D$: ${\Gamma}=\Gamma_{D}=
    2\beta\circ\rho-I$ and, consequently, $h_{\Gamma}=h_{D}=\beta\circ\rho$, $\,v_{\Gamma}=v_{D}=\gamma\circ K$.
\end{lem}

\begin{defn}\label{hor.rec.} Let the pullback bundle $\pi^{-1}(TM)$  be equipped  with a metric tensor $g$.
A regular connection $\overline{D}$ on $\pi^{-1}(TM)$ is said to be horizontally
recurrent if there exists a scalar 1-form $A$ on $\pi^{-1}(TM)$ such that
$$\overline{D}_{{\bar{\beta}} \overline{X}} g=A(\overline{X})g \quad \forall \,  \overline{X}\in  \cp,$$
where $\bar{\beta}$ is the horizontal map of $\overline{D}$. The scalar form $A$ is called the h-recurrence form of $\overline{D}$.
\end{defn}

Let $(M,L)$ be a Finsler manifold. Let $\nabla$ be the Cartan connection associated with $(M,L)$. We denote by $K,\beta,T$ the connection map, the horizontal map and the (h)hv-torsion of $\nabla$, respectively. We also denote by $R,P,S$ the h-, hv- and  v-curvature tensors of $\nabla$, respectively.
Now, we announce the main result of the present section.
\begin{thm} \label{th.v1}  Let $(M,L)$ be a Finsler manifold and $g$ the Finsler metric defined by $L$.
For each  scalar ($\pi$)1-form $A$, there exists a unique regular connection $\overline{D}$ on $\pi^{-1}(TM)$ such
that
\begin{description}
  \item[(C1)] $\overline{D}$ is horizontally recurrent with h-recurrence form $A$\,{\em:}  $(\overline{D}g) \circ \bar{\beta}=A \otimes g$,
  \item[(C2)] the metric $g$ is  $\overline{D}$-vertically parallel\,{\em:}  $\overline{D}_{\gamma \overline{X}} g=0$,
  \item[(C3)] the (h)h-torsion $\overline{Q}$ of $\overline{D}$ vanishes\,{\em:} $\overline{Q}=0$,
  \item[(C4)] the (h)hv-torsion $\overline{T}$ of $\overline{D}$ satisfies   $g(\overline{T}(\overline{X},\overline{Y}),\overline{Z}) =g(\overline{T}(\overline{X},\overline{Z}),\overline{Y})$.
\end{description}
Such a connection is called the horizontally recurrent Finsler (HRF-) connection with h-recurrence form $A$.
\end{thm}

\begin{proof}
 First we prove the {\it \textbf{uniqueness}}. Since $\overline{D}$ is a regular connection, then, by Definition \ref{sem.1},
 its horizontal (vertical)  projector is given by  $\bar{h}:=\bar{\beta} \circ \rho$ ($\bar{v}:=I-\bar{\beta} \circ \rho $).
  On the other hand, from axioms  (C2) and (C4), taking into account Lemma 4 of \cite{r92},
 we deduce that $\overline{T}(\overline{X},\overline{\eta})=0$ for all  $\overline{X}\in \mathfrak{X}(\pi (M))$. Consequently, using
 Lemma \ref{eqv.},  it follows that   $\overline{K}=\gamma^{-1}$ on $V (T M)$ and the associated  nonlinear
connection $\Gamma_{\overline{D}}$ is given by
$\Gamma_{\overline{D}}=\bar{\beta}\circ\rho -\gamma\circ \overline{K}$, where $\overline{K}$  is the connection  map of $\overline{D}$.
Moreover,  the horizontal
and vertical projectors of $\Gamma_{\overline{D}}$ are given by $\bar{h}=\bar{\beta}
\circ \rho$ and $\bar{v}=\gamma \circ \overline{K}$, respectively.

 \par
In view of axioms (C2) and (C4), one can show that, for all  $\overline{X}, \overline{Y}, \overline{Z}\in \mathfrak{X}(\pi (M))$,
\begin{eqnarray}\label{eq.t1}
2g(\overline{D} _{\gamma \overline{X}} \overline{Y}, \overline{Z})&=&\gamma \overline{X}\cdot g(\overline{Y},\overline{Z})
 -g(\overline{Y}, \overline{T}(\overline{X},\overline{Z}))+g(\overline{Z}, \overline{T}(\overline{X},\overline{Y})) \nonumber \\
    &&+g(\overline{Y},\rho[\bar{\beta} \overline{Z}, \gamma \overline{X}])+g(\overline{Z},\rho[\gamma \overline{X},\bar{\beta} \overline{Y}]).
\end{eqnarray}
As the difference between $\bar{\beta} \overline{X}$ and ${\beta} \overline{X}$ is vertical, one can set
\begin{equation}\label{diff}
 \bar{\beta} \overline{X}={\beta} \overline{X}+\gamma \overline{X}_{t},
\end{equation}
for some $\overline{X}_{t} \in \cp$.  Since $J^{2}=0, \, [J,J]=0$ and $\rho \circ J=0$, we get
\begin{equation}\label{eq.t2}
  \rho[\gamma \overline{X},\bar{\beta}\, \overline{Y}]=\rho[\gamma \overline{X},{\beta}\, \overline{Y}].
\end{equation}
Making use of (\ref{eq.t2}) and Theorem 4\textbf{(a)} of \cite{r92}, Equation (\ref{eq.t1}) implies that \vspace{-0.2cm}
\begin{equation}\label{11}
 \overline{D} _{\gamma \overline{X}} \overline{Y}=\nabla  _{\gamma \overline{X}} \overline{Y}.\vspace{-0.1cm}
\end{equation}
\par
Similarly, using axioms  (C1) and (C3), we obtain,  for all  $\overline{X}, \overline{Y}, \overline{Z}\in \mathfrak{X}(\pi (M))$,
\begin{eqnarray}\label{eq.t3}
2g(\overline{D} _{\bar{\beta} \overline{X}} \overline{Y}, \overline{Z})&=&\bar{\beta} \overline{X}\cdot g(\overline{Y},\overline{Z})
+\bar{\beta} \overline{Y}\cdot g(\overline{Z},\overline{X})-\bar{\beta} \overline{Z}\cdot g(\overline{X},\overline{Y}) \nonumber\\
&&-g(\overline{X},\rho[\bar{\beta} \overline{Y},\bar{\beta} \overline{Z}])-g(\overline{Y},\rho[\bar{\beta} \overline{Z},\bar{\beta} \overline{X}])
+g(\overline{Z},\rho[\bar{\beta} \overline{X},\bar{\beta} \overline{Y}]) \nonumber\\
&&-A(\overline{X})g(\overline{Y},\overline{Z})-A(\overline{Y})g(\overline{Z},\overline{X})+A(\overline{Z})g(\overline{X},\overline{Y}).
\end{eqnarray}
From (\ref{diff}), it is easy to show that
\begin{equation*}
  \rho[\bar{\beta} \overline{X},\bar{\beta} \overline{Y}]=\rho[\beta \overline{X},{\beta} \overline{Y}]+
  \rho[\beta \overline{X},{\gamma} \overline{Y}_{t}]+\rho[\gamma \overline{X}_{t},{\beta} \overline{Y}],
\end{equation*}
Using the above relation, Proposition \ref{prop.1}, Proposition \ref{prop.2} and  Theorem 4\textbf{(b)} of \cite{r92}, (\ref{eq.t3})  becomes
\begin{eqnarray}\label{eq.t4}
2g(\overline{D} _{\bar{\beta} \overline{X}} \overline{Y}, \overline{Z})&=&2g(\nabla _{{\beta} \overline{X}} \overline{Y}, \overline{Z})
-A(\overline{X})g(\overline{Y},\overline{Z})-A(\overline{Y})g(\overline{Z},\overline{X})\nonumber\\
&&+A(\overline{Z})g(\overline{X},\overline{Y}) +2g(D^{\circ} _{{\gamma} \overline{X}_{t}} \overline{Y}, \overline{Z})+({D^{\circ}} _{\gamma \overline{X}_{t}}g)( \overline{Y}, \overline{Z})\nonumber\\
&&+({D^{\circ}} _{\gamma \overline{Y}_{t}}g)( \overline{Z}, \overline{X})-({D^{\circ}} _{\gamma \overline{Z}_{t}}g)( \overline{X}, \overline{Y})\nonumber\\
&=&2g(\nabla _{{\beta} \overline{X}} \overline{Y}, \overline{Z})
-A(\overline{X})g(\overline{Y},\overline{Z})-A(\overline{Y})g(\overline{Z},\overline{X})\nonumber\\
&&+A(\overline{Z})g(\overline{X},\overline{Y}) +2g(D^{\circ} _{{\gamma} \overline{X}_{t}} \overline{Y}, \overline{Z})+ 2\textbf{T}( \overline{X}_{t},\overline{Y}, \overline{Z})\nonumber\\
&& + 2T( \overline{Y}_{t},\overline{Z}, \overline{X})-2T( \overline{Z}_{t},\overline{X}, \overline{Y}).
\end{eqnarray}
Setting $\overline{X}=\overline{Y}=\overline{\eta}$ in (\ref{eq.t4}),  noting that $\bar{K} \circ \bar{\beta}=K \circ \beta=0$
and $i_{\overline{\eta}}T=0$, we get
\begin{equation*}\label{B}
2\overline{\eta}_{t}=2A(\overline{\eta})\overline{\eta}-L^{2}\overline{a},
\end{equation*}
where $g(\overline{a},\overline{X}):=A(\overline{X})$. From which,  by setting again $\overline{Y}=\overline{\eta}$ into Equation
 (\ref{eq.t4}), we obtain
 \begin{equation}\label{B1}
   \bar{\beta} \overline{X}-\beta \overline{X}=\gamma\overline{X}_{t}=\frac{1}{2}\{A(\overline{X})\gamma\overline{\eta}+A(\overline{\eta})\gamma\overline{X}-g(\overline{X},\overline{\eta})\, \gamma\overline{a} +L^{2}\gamma T(\overline{a},\overline{X})\}.
 \end{equation}
   From which, one can show that
   \begin{equation}\label{tt1}
   2T(\overline{X}_{t},\overline{Y}, \overline{Z})=A(\overline{\eta})T(\overline{X},\overline{Y}, \overline{Z})
   -g(\overline{X},\overline{\eta} )T(\overline{a},\overline{Y}, \overline{Z})+L^{2}T(
   T(\overline{a},\overline{X}),\overline{Y}, \overline{Z}),
   \end{equation}
    \begin{equation}\label{tt2}
      2{D^{\circ}} _{\gamma \overline{X}_{t}}\overline{Y}= A(\overline{X}){D^{\circ}} _{\gamma \overline{\eta}}\overline{Y}+
   A(\overline{\eta}){D^{\circ}} _{\gamma \overline{X}}\overline{Y}-g(\overline{X},\overline{\eta}){D^{\circ}} _{\gamma \overline{a}}\overline{Y}
   +L^{2}{D^{\circ}} _{\gamma T(\overline{a},\overline{X})}\overline{Y}.
   \end{equation}
Now, using (\ref{B1}), (\ref{tt1}) and (\ref{tt2}), Equation (\ref{eq.t4})  reduces to
 \begin{eqnarray}\label{22}
\overline{D} _{{\bar{\beta}} \overline{X}} \overline{Y}&=&\nabla _{{\beta} \overline{X}} \overline{Y}
+\frac{1}{2}\{A(\overline{\eta})T(\overline{X},\overline{Y})
-A(\overline{X})\overline{Y}-A(\overline{Y})\overline{X}+\overline{a}g(\overline{X},\overline{Y})\nonumber\\
&&+A(\overline{X}){D^{\circ}} _{\gamma \overline{\eta}}\overline{Y}+
   A(\overline{\eta}){D^{\circ}} _{\gamma \overline{X}}\overline{Y}-g(\overline{X},\overline{\eta}){D^{\circ}} _{\gamma \overline{a}}\overline{Y}
   +L^{2}{D^{\circ}} _{\gamma T(\overline{a},\overline{X})}\overline{Y}\nonumber\\
&&-L\ell(\overline{X})T(\overline{a},\overline{Y})-L\ell(\overline{Y})T(\overline{a},\overline{X})+
T(\overline{a},\overline{X},\overline{Y})\overline{\eta}
\\
&&+L^{2}[{T}(T(\overline{a},\overline{X}),\overline{Y})+{T}(T(\overline{a},\overline{Y}),\overline{X})
-{T}(T(\overline{X},\overline{Y}),\overline{a})]\} \nonumber.
\end{eqnarray}
Consequently, from (\ref{11}),  (\ref{22}) and taking into account (\ref{B1}), the full
expression of $\overline{D}_{X}\overline{Y}$ is given by:
\begin{eqnarray}\label{33}
\overline{D} _{X} \overline{Y}&=&\nabla _{X} \overline{Y}
+\frac{1}{2}\{g(\rho {X},\overline{Y})\overline{a}-A(\rho {X})\overline{Y}-A(\overline{Y})\rho {X}-L\ell(\overline{Y})T(\overline{a},\rho {X})\nonumber\\
&&+T(\overline{a},\rho {X},\overline{Y})\overline{\eta}+L^{2}[{T}(T(\overline{a},\overline{Y}),\rho {X})-{T}(T(\rho {X},\overline{Y}),\overline{a})]\}.
\end{eqnarray}
Hence ${\overline{D}}_{X}\overline{Y}$ is uniquely determined
by the right-hand side of (\ref{33}).

  \vspace{7pt}
\par Now, we prove the {\it
\textbf{existence}} of $\overline{D}$. For a given 1-scalar form $A$ on $\pi^{-1}(TM)$,  we define
$\overline{D}$ by the requirement that (\ref{33}) holds, or equivalently,   (\ref{11}), (\ref{22})  and (\ref{B1}) hold.
Then, using the properties of Cartan connection   $\nabla$ and the results of \cite{r92}, it is not difficult to show that the connection
$\overline{D}$ satisfies the conditions (C1)-(C4).
\end{proof}

\begin{rem} {\emph {If in the above theorem we take $A=0$, the connection  $\overline{D}$  reduces to the Cartan connection $\nabla$. Consequently, the above theorem  generalizes the existence and uniqueness theorem of Cartan connection \cite{r92}.}}
\end{rem}

\begin{cor}\label{th.v3} The  HRF-connection  $\overline{D}$ and  the Cartan connection $\nabla$ are related by
$$\overline{D} _{X} \overline{Y}=\nabla _{X} \overline{Y}+ N(\rho X, \overline{Y});$$
    \begin{eqnarray*}\label{B322}
 N(\rho X, \overline{Y})&:=&\frac{1}{2}\{g(\rho {X},\overline{Y})\overline{a}-A(\rho {X})\overline{Y}-A(\overline{Y})\rho {X}-L\ell(\overline{Y})T(\overline{a},\rho {X}) \nonumber\\
&&+T(\overline{a},\rho {X},\overline{Y})\overline{\eta}
+L^{2}[{T}(T(\overline{a},\overline{Y}),\rho {X})-{T}(T(\rho {X},\overline{Y}),\overline{a})]\}.
\end{eqnarray*}
\end{cor}

In view of Theorem \ref{th.v1}, we have
\begin{prop} \label{th.v2} ~ \par
  \begin{description}
  \item[(a)] The  spray $\overline{G}$  associated with  $\overline{D}$
 is related to the canonical spray $G$ by {\em:}
 \begin{equation*}
 \overline{G}=G+A(\overline{\eta}) \mathcal{C}-\frac{1}{2}L^{2}\gamma \overline{a}.
\end{equation*}

  \item[(b)]The nonlinear connection  $\overline{\Gamma}$  associated with  $\overline{D}$
 is related to the Barthel connection $\Gamma$ by {\em:}
  \begin{equation*}
   \overline{\Gamma}(X)= \Gamma(X) +\{A( \rho {X})\mathcal{C}+A(\overline{\eta})J {X}-g(\rho {X},\overline{\eta})\,\gamma \overline{a} +L^{2} \gamma T(\overline{a},\rho {X})\}.
 \end{equation*}
\end{description}

\end{prop}

\begin{prop}\label{7.pp.7}
Let $\overline{S}$,  $\overline{P}$ and $\overline{R}$ be the v-, hv- and h-curvatures of HRF-connection  $\overline{D}$,
then we have~\footnote{$\mathfrak{U}_{X, Y}\{B(X,Y)\}:=B(X,Y)-B(Y,X)$.} \vspace{-0.1cm}
\begin{description}

  \item[(a)]${{\overline{S}}}(\overline{X},\overline{Y})\overline{Z}= S(\overline{X},\overline{Y})\overline{Z}$.

  \item[(b)]${{\overline{P}}}(\overline{X},\overline{Y})\overline{Z}= P(\overline{X},\overline{Y})\overline{Z}
  +  (\nabla_{\gamma \overline{Y}}N)(\overline{X},\overline{Z})+N({T}(\overline{Y},\overline{X}),\overline{Z})$\\
  ${\qquad\qquad\ \ \ }+\frac{1}{2}\{A(\overline{\eta})S(\overline{X},\overline{Y})\overline{Z}-L\ell(\overline{X})
  S(\overline{a},\overline{Y})\overline{Z}+L^{2}S(T(\overline{a},\overline{X}),\overline{Y})\overline{Z} \}$.

  \item[(c)]${{\overline{R}}}(\overline{X},\overline{Y})\overline{Z}= R(\overline{X},\overline{Y})\overline{Z}
  +P(\overline{X},\overline{Y}_{t})\overline{Z}-P(\overline{Y},\overline{X}_{t})\overline{Z}+S(\overline{X}_{t},\overline{Y}_{t})\overline{Z}$\\
    ${\qquad\qquad\ \ \ }
   +\mathfrak{U}_{\overline{X}, \overline{Y}}\{(\nabla_{\bar{\beta}\overline{Y}}N)(\overline{X},\overline{Z})
   +N(\overline{Y}, N(\overline{X}, \overline{Z}))+N(T(\overline{Y}_{t},\overline{X}),\overline{Z})\}$,\\
   where $\overline{Y}_t$ is given by (\ref{B1}).
\end{description}
\end{prop}

\section{Special HRF-connection}

In this section, we investigate a special horizontally recurrent Finsler connection for which the h-recurrence form $A$
is taken to be  $\ell:=L^{-1}i_{\overline{\eta}}\,g$.
In what follows $\overline{D}$ will denote the HRF-connection whose h-recurrence form is $\ell$, and will be called the special HRF-connection. In this case $\overline{a}=L^{-1} \, \overline{\eta}$.

\bigskip
The following two lemmas are useful for subsequence use.
\begin{lem} \label{th.v22}  The nonlinear connection  $\overline{\Gamma}$  associated with the special HRF-connection  $\overline{D}$
  is given by
  $  \overline{\Gamma}= \Gamma +L J.$ Consequently,
   ${\bar{\beta}}=\beta+\frac{1}{2}L \gamma$ and
 ${\overline{K}}=K-\frac{1}{2}L \rho$.
     \end{lem}

\begin{proof} The proof follows from Proposition \ref{th.v2} and the identities  $\overline{a}=L^{-1} \overline{\eta}$,
$\ell(\overline{\eta})=L$ and $i_{\overline{\eta}}T=0$.
\end{proof}

\begin{lem}\label{th.v33} For the special HRF-connection  $\overline{D}$, we have
  \begin{description}
   \item[(a)] $N(\overline{X},\overline{Y})=\frac{1}{2}\{L^{-1}g(\overline{X},\overline{Y})
       \overline{\eta}-\ell(\overline{X})\overline{Y}-\ell(\overline{Y})\overline{X}\} $.
      \item[(b)] $N(\overline{X},\overline{\eta})=-\frac{1}{2}L \overline{X}$.
   \item[(c)]$(\nabla_{\beta \overline{X}} N)(\overline{Y},\overline{Z})=0$.
   \item[(d)] $(\nabla_{\gamma \overline{X}} N)(\overline{Y},\overline{Z})
   =\frac{1}{2L}\{g(\overline{Y},\overline{Z})\phi(\overline{X})
   -\hbar(\overline{X},\overline{Y})\overline{Z}-\hbar(\overline{X},\overline{Z})\overline{Y}\}$.
   \item[(e)] $(\nabla_{\bar{\beta} \overline{X}} N)(\overline{Y},\overline{Z})
   =\frac{1}{4}\{g(\overline{Y},\overline{Z})\phi(\overline{X})
   -\hbar(\overline{X},\overline{Y})\overline{Z}-\hbar(\overline{X},\overline{Z})\overline{Y}\}$,
 \end{description}
 where $\phi (\overline{X}):=\overline{X}-L^{-1}\ell(\overline{X}) \overline{\eta} $ and \,$\hbar(\overline{X},\overline{Y})=g(\phi(\overline{X}),\overline{Y})$ is the angular metric.
 \end{lem}
\begin{proof} The proof follows from Corollary  \ref{th.v3} and fact that $T$ is symmetric and indicatory, taking the identities $\nabla g=0$,
 $\nabla_{\beta \overline{X}} L=0$,  $\nabla_{\beta \overline{X}}\ell=0 $, $\nabla_{\gamma \overline{X}} L=\ell(\overline{X})$ and  $(\nabla_{\gamma \overline{X}}\ell)(\overline{Y})=L^{-1}\hbar(\overline{X},\overline{Y})$ into account.
\end{proof}

\begin{thm}\label{th.v3th} The (v)hv-torsion tensor
 $\widehat{\overline{P}}$ of the special HRF-connection $\overline{D}$ never vanishes.
    \end{thm}
\begin{proof}  From Proposition \ref{7.pp.7}, taking into account Lemma \ref{th.v22}, Lemma \ref{th.v33} and the fact that
 $T$ and $S$ are indicatory, one can show that
 \begin{eqnarray*}
   \overline{P}(\overline{X},\overline{Y})\overline{Z} &=& P(\overline{X},\overline{Y})\overline{Z}+\frac{1}{2}\{L S(\overline{X},\overline{Y})\overline{Z}
    +T(\overline{X},\overline{Y},\overline{Z})\frac{\overline{\eta}}{L}-\ell(\overline{Z})T(\overline{X},\overline{Y})\}\\
    & & -\frac{1}{2L}\{\hbar(\overline{Y},\overline{Z})\overline{X} + \hbar(\overline{Y},\overline{X})\overline{Z}-g(\overline{X},\overline{Z})\phi(\overline{Y})\}.
 \end{eqnarray*}
 From wihch
 \begin{equation}\label{pp}
   \widehat{\overline{P}}(\overline{X},\overline{Y})
   =\widehat{P}(\overline{X},\overline{Y})-\frac{1}{2}\{L T(\overline{X},\overline{Y})
       -\ell(\overline{X})\overline{Y}+L^{-1}\,g(\overline{X},\overline{Y})\overline{\eta}\}.
 \end{equation}
 \par  Now, assume  the contrary, that is $\widehat{\overline{P}}=0$.
  Hence, setting $\overline{X}=\overline{\eta}$ in (\ref{pp}), using the fact that $\widehat{P}$ and $T$ are indicatory, we obtain
 $$\frac{1}{2} L \, \phi(\overline{Y})=0 \quad \forall\, \overline{Y}\in\cp.$$
  Consequently, $(n-1)L=0$ (since $Tr \, \phi=n-1$), which is a contradiction.
\end{proof}

\begin{cor} In view of Equation \emph{(}\ref{pp}\emph{)}, we have
\begin{description}
  \item[(a)]$ \widehat{\overline{P}}(\overline{X},\overline{Y})-\widehat{\overline{P}}(\overline{Y},\overline{X})
   =\frac{1}{2}\{\ell(\overline{X})\overline{Y}-\ell(\overline{Y})\overline{X}\}.$
  \item[(b)] $g(\widehat{\overline{P}}(\overline{X},\overline{Y}),\overline{\eta})=-\frac{1}{2}\hbar(\overline{X},\overline{Y}).$
\end{description}
  \end{cor}

\begin{prop}\label{th.v3bb} The h-curvature tensor
 ${\overline{R}}$ of the special HRF-connection $\overline{D}$ has the form
 \begin{eqnarray*}
   \overline{R}(\overline{X},\overline{Y})\overline{Z} &=& R(\overline{X},\overline{Y})\overline{Z}+\frac{1}{4}L^{2} S(\overline{X},\overline{Y})\overline{Z}
  +\frac{1}{2}L\{P(\overline{X},\overline{Y})\overline{Z}-P(\overline{Y},\overline{X})\overline{Z}\}\\
    & & +\frac{1}{4}\{g(\overline{X},\overline{Z})\overline{Y} - g(\overline{Y},\overline{Z})\overline{X}\}.
 \end{eqnarray*}
  \end{prop}
\begin{proof} The proof follows from Proposition \ref{7.pp.7}, Lemma \ref{th.v22} and Lemma \ref{th.v33}.
\end{proof}

\begin{defn}\label{def.2}{\em{\cite{r86}}} A Finsler manifold $(M,L)$, with $dim\, M\geq 3$,
is said to be:
\begin{description}
  \item[(a)] $h$-isotropic with  a scalar
$k_{o}$ if  the h-curvature tensor $R$  of Cartan connection has the form
$R(\overline{X},\overline{Y})\overline{Z}=k_{o} \{g(\overline{X},\overline{Z})
\overline{Y}-g(\overline{Y},\overline{Z})\overline{X} \}.$
  \item[(b)] of constant curvature $k$ if the (v)h-torsion tensor $\widehat{R}$ of Cartan connection satisfies
the relation
 $\widehat{R}(\overline{\eta},\overline{X})= k L^{2} \phi(\overline{X}).$
  \end{description}
\end{defn}

\begin{defn}\label{def.3}{\em{\cite{r86}}} A Finsler manifold $(M,L)$ is said to be P-symmetric if the mixed curvature  tensor $P$ of Cartan connection satisfies
$P(\overline{X},\overline{Y})\overline{Z}
 =P(\overline{Y},\overline{X})\overline{Z}$, for all  $\overline{X},\overline{Y}, \overline{Z} \in\cp$.
\end{defn}

\begin{thm} Let $(M,L)$ be an h-isotropic Finsler manifold with scalar $(-\frac{1}{4})$. The h-curvature tensor
 ${\overline{R}}$ of the special HRF-connection $\overline{D}$  vanishes if and only  if the v-curvature tensor $S$ of Cartan connection has property that
  $\nabla_{\beta \overline{\eta}} \,S=\frac{L}{2}S$.
    \end{thm}
\begin{proof} The proof follows from proposition  \ref{th.v3bb}, Lemma \ref{th.v22} and the fact that \cite{r96}
\begin{equation}\label{pp2}
  P(\overline{X},\overline{Y})\overline{Z}-P(\overline{Y},\overline{X})\overline{Z}=-(\nabla_{\beta \overline{\eta}} \,S)(\overline{X},\overline{Y},\overline{Z}).
\end{equation}
\end{proof}

\begin{thm} Let $(M,L)$ be a $P$-symmetric h-isotropic  Finsler manifold with scalar $(-\frac{1}{4})$. The h-curvature tensor
 ${\overline{R}}$ of the special HRF-connection $\overline{D}$  vanishes if and only  if the v-curvature tensor $S$ vanishes.
\end{thm}
\begin{proof} The proof follows from Proposition \ref{th.v3bb} and Relation (\ref{pp2}).
\end{proof}

\begin{thm} If  the (v)h-torsion tensor
 ${\widehat{\overline{R}}}$ of  $\overline{D}$  vanishes, then $(M,L)$ is of constant curvature.
     \end{thm}
\begin{proof} From proposition  \ref{th.v3bb}, we obtain
 $$ \widehat{\overline{R}}(\overline{X},\overline{Y})=\widehat{R}(\overline{X},\overline{Y})+\frac{L}{4}\{\ell(\overline{X}) \overline{Y}-
 \ell(\overline{Y}) \overline{X}\}.$$
 Now, if the (v)h-torsion tensor  ${\widehat{\overline{R}}}$ of  $\overline{D}$  vanishes, the above relation implies that
    $$ \widehat{R}(\overline{\eta},\overline{Y})=-\frac{1}{4} \, L^{2}\phi(\overline{Y}).$$
 Hence, by Definition \ref{pp}\textbf{(b)},  the result follows.
 \end{proof}


\providecommand{\bysame}{\leavevmode\hbox
to3em{\hrulefill}\thinspace}
\providecommand{\MR}{\relax\ifhmode\unskip\space\fi MR }
\providecommand{\MRhref}[2]{%
  \href{http://www.ams.org/mathscinet-getitem?mr=#1}{#2}
} \providecommand{\href}[2]{#2}

\end{document}